\documentclass[10pt]{amsart}
\usepackage{amsmath,amssymb,amsthm,graphicx,mathrsfs,url}
\usepackage[usenames,dvipsnames]{color}
\definecolor{darkred}{rgb}{0.6,0.1,0.1}
\definecolor{darkred2}{rgb}{0.6,0,0}
\definecolor{darkblue}{rgb}{0.1,0.1,0.4}
\definecolor{darkgrey}{rgb}{0.5,0.5,0.5}
\numberwithin{equation}{section}
\theoremstyle{plain}% default
\newtheorem{thm}{Theorem}[section]
\newtheorem{lem}[thm]{Lemma}

\newtheorem{prop}[thm]{Proposition}

\newtheorem*{cor*}{Corollary}

\newtheorem*{thmA}{Theorem A}

\theoremstyle{remark}

\newtheorem{example}[thm]{Example}

\theoremstyle{plain}

\theoremstyle{definition}
\newtheorem{dfn}[thm]{Definition}

\renewcommand{\Re}{{\mathrm{Re}}}

\newcommand{\be}{\begin{equation}}
\newcommand{\ee}{\end{equation}}
\newcommand{\beu}{\begin{equation*}}
\newcommand{\eeu}{\end{equation*}}
\newcommand{\besu}{\begin{equation*}
\begin{aligned}}
\newcommand{\eesu}{\end{aligned}
\end{equation*}}
\newcommand{\bes}{\begin{equation}
\begin{aligned}}
\newcommand{\ees}{\end{aligned}
\end{equation}}

\newcommand\fra{\mathfrak a}

\newcommand\ov{\overline}
\newcommand\wt{\widetilde}

\newcommand\sess{\sigma_{\rm ess}}

\newcommand\void[1]{}

\def\sess{\sigma_{\rm ess}}

      \def\dC{{\mathbb C}}

   \def\dN{{\mathbb N}}   
      \def\dR{{\mathbb R}}

\newcommand{\dom}{\mathrm{dom}\,}

\begin{document}
\title[An eigenvalue inequality for Schr\"odinger operators]{An eigenvalue inequality for Schr\"odinger operators
with {\boldmath$\delta$} and {\boldmath$\delta'$}-interactions supported on hypersurfaces}

\author{Vladimir Lotoreichik}

\author{Jonathan Rohleder}

\address{Institut f\"{u}r Numerische Mathematik\\
Technische Universit\"{a}t Graz\\
 Steyrergasse 30, A 8010 Graz}
\email{lotoreichik@math.tugraz.at \and rohleder@tugraz.at}

\begin{abstract}
We consider self-adjoint Schr\"odinger operators in~$L^2 (\dR^d)$
% $-\Delta_{\delta,\alpha}$ and $-\Delta_{\delta',\beta}$ 
with a $\delta$-interaction of strength~$\alpha$ and a $\delta'$-interaction of strength $\beta$, respectively, supported on a hypersurface, where $\alpha$ and $\beta^{-1}$ are bounded, real-valued functions. It is known that the inequality~$0 < \beta \leq 4/\alpha$ implies inequality of the eigenvalues of 
% $- \Delta_{\delta', \beta}$ and $- \Delta_{\delta, \alpha}$ 
these two operators below the bottoms of the essential spectra. We show that this eigenvalue inequality is strict whenever~$\beta < 4 / \alpha$ on a nonempty, open subset of the hypersurface. Moreover, we point out special geometries of the interaction support, such as broken lines or infinite cones, for which strict inequality of the eigenvalues even holds in the borderline case $\beta = 4 / \alpha$.
\end{abstract}


\maketitle

%**********************************************************
\section{Introduction}
%**********************************************************

Schr\"odinger operators with $\delta$ and $\delta'$-interactions supported on hypersurfaces have attracted considerable attention in recent years, see the review paper~\cite{E08} and, e.g.,~\cite{BEL13, EI01, EJ13, EP14}, as well as~\cite{BEL14,BEW09,CDR08,DR13,EN03,L13} for interactions supported on hypersurfaces with special geometries. In this note we focus on the self-adjoint Schr\"odinger operators~$-\Delta_{\delta,\alpha}$ and~$-\Delta_{\delta',\beta}$ in $L^2(\dR^d)$, $d \geq 2$, which are formally given by
\begin{align*}
 - \Delta_{\delta, \alpha} = - \Delta - 
 \alpha \langle \cdot, \delta_\Sigma  \rangle \delta_\Sigma \quad \text{and} \quad - \Delta_{\delta', \beta} = - \Delta - \beta \langle \cdot,\delta'_\Sigma \rangle \delta'_\Sigma,
\end{align*}
where $\Delta$ is the Laplacian and the support $\Sigma$ of the interactions is a Lipschitz hypersurface; we emphasize that $\Sigma$ is not required to be compact or connected, see Section~\ref{sec:weak} for the details. These operators can be defined rigorously, e.g., via quadratic forms, as is indicated in Section~\ref{sec:def} below. We assume that the strengths $\alpha$ and $\beta$ of the interactions are real-valued functions on $\Sigma$ with $\alpha, \beta^{-1} \in L^\infty (\Sigma)$. 

Let us denote by $\sigma_{\rm ess} (- \Delta_{\delta, \alpha})$ and $\sigma_{\rm ess} (- \Delta_{\delta', \beta})$ the essential spectra of $- \Delta_{\delta, \alpha}$ and $- \Delta_{\delta', \beta}$, respectively. Moreover, let
\[
\lambda_1(-\Delta_{\delta,\alpha}) \le 
\lambda_2(-\Delta_{\delta,\alpha})\le \dots <
\inf \sess (-\Delta_{\delta, \alpha})
\]
and 
\[
\lambda_1(-\Delta_{\delta',\beta}) \le 
\lambda_2(-\Delta_{\delta',\beta})\le \dots <
\inf \sess (-\Delta_{\delta', \beta})
\]
be the eigenvalues of $- \Delta_{\delta, \alpha}$ and $- \Delta_{\delta', \beta}$, respectively, below the bottom of the essential spectrum, counted with multiplicities; for many choices of~$\Sigma$ the existence of such eigenvalues has been proved, see e.g. \cite{BEL13, BEL14, BEW09, EI01, EK03}. 
% We denote by $N(-\Delta_{\delta,\alpha})$ and $N(-\Delta_{\delta',\beta})$ the number of eigenvalues of $-\Delta_{\delta,\alpha}$ and $-\Delta_{\delta', \beta}$, respectively, below the bottom of the essential spectrum.

In~\cite[Theorem~3.6]{BEL13} for $0 < \beta \leq \frac{4}{\alpha}$ the operator inequality
\begin{align*}
% \label{eq:opEq}
 U^{-1} \big( - \Delta_{\delta', \beta} \big) U \leq - \Delta_{\delta, \alpha}
\end{align*}
was established, where $U$ is a unitary transformation in $L^2 (\dR^d)$; cf.~\eqref{U} below. This implies $\inf \sess (- \Delta_{\delta', \beta}) \leq \inf \sess(-\Delta_{\delta,\alpha})$ as well as 
\begin{align*}
 \lambda_n (-\Delta_{\delta', \beta}) \leq \lambda_n (-\Delta_{\delta,\alpha})
\end{align*}
for all $n \in \dN$ such that $\lambda_n (- \Delta_{\delta, \alpha}) < \inf \sess (- \Delta_{\delta', \beta})$.
The aim of this note is to sharpen the latter inequality as follows.

%----------------------------------------------------------
\begin{thmA}\label{thm}
Let $0 < \beta \leq 4 / \alpha$ and assume that $\beta |_\sigma < 4 / \alpha |_\sigma$ on a nonempty, open set $\sigma \subset \Sigma$. Then
\[
\lambda_n(-\Delta_{\delta', \beta}) < \lambda_n(-\Delta_{\delta,\alpha})
\]
holds for all $n \in \dN$ such that $\lambda_n (- \Delta_{\delta, \alpha}) < \inf \sess (- \Delta_{\delta', \beta})$.
\end{thmA}
%----------------------------------------------------------

If the hypersurface~$\Sigma$ is compact, it is known that $\sess (- \Delta_{\delta', \beta}) =  \sess(-\Delta_{\delta,\alpha}) = [0, \infty)$; cf.~\cite[Theorem~4.2]{BEL13}. Therefore in this case Theorem~A implies strict inequality between all negative eigenvalues of $-\Delta_{\delta,\alpha}$ and $-\Delta_{\delta',\beta}$; note that if $\Sigma$ is compact and sufficiently regular, these operators have only finitely many negative eigenvalues, see~\cite[Theorem 4.2]{BEKS94} and~\cite[Theorem 3.14]{BLL13}.

%----------------------------------------------------------
\begin{cor*}
Let the assumptions of Theorem~A be satisfied and let, additionally,~$\Sigma$ be compact and $C^\infty$-smooth. Then 
\[
\lambda_n(-\Delta_{\delta', \beta}) < \lambda_n(-\Delta_{\delta,\alpha}), \quad n = 1, \dots, N (- \Delta_{\delta, \alpha}),
\]
holds, where $N (- \Delta_{\delta, \alpha})$ denotes the number of negative eigenvalues of~$- \Delta_{\delta, \alpha}$.
\end{cor*}
%----------------------------------------------------------

Our proof of Theorem~A is based on an idea which was suggested by Filonov in~\cite{F05} and which was used and modified later on in various spectral problems, see~\cite{FL10, GM09, K10, R14}. We remark that the result of Theorem~A can be proved analogously for the more general case of $\Sigma$ being a Lipschitz partition of~$\dR^d$ as considered in~\cite{BEL13}. However, in order to avoid technicalities we restrict ourselves to the case of a hypersurface.

Besides the general result of Theorem~A, which is proved in Section~\ref{sec:proof}, in Section~\ref{sec:examples} we discuss several examples of special geometries of $\Sigma$ for which the strict inequality of Theorem~A holds even in the borderline case $\beta = 4 / \alpha$, for constant strengths~$\alpha, \beta$. 
Among these examples there are the cases of a broken line in $\dR^2$ and an infinite cone in~$\dR^3$. 
% In these examples strict eigenvalue inequalities are, roughly speaking, induced by the special geometry of $\Sigma$.

%-------------------------------------------------------------

%**********************************************************
\section{Preliminaries}\label{sec:prel}
%**********************************************************

%-------------------------------------------------------------
\subsection{Lipschitz hypersurfaces and weak normal derivatives}\label{sec:weak}
%-------------------------------------------------------------

Let us first recall some basic facts and notions. For an arbitrary open set $\Omega \subset \dR^d$, $d \geq 2$, we write $(\cdot, \cdot)_{\Omega}$ for both the inner products in the spaces $L^2 (\Omega)$ and $L^2 (\Omega, \dC^d)$ of scalar and vector-valued square-integrable functions, respectively, without any danger of confusion; the associated norms are denoted by~$\| \cdot \|_\Omega$. As usual, $H^1(\Omega)$ is the Sobolev space of order one and $H_0^1 (\Omega)$ denotes the closure of the space of smooth functions with compact supports in $H^1 (\Omega)$. 

In the following we understand Lipschitz domains in the general sense of, e.g., \cite[\S VI.3]{St}; in particular, we allow noncompact boundaries. We write $\Sigma$ for the boundary of a Lipschitz domain $\Omega$ and denote the inner product in $L^2 (\Sigma)$ by $(\cdot, \cdot)_\Sigma$ and the corresponding norm by $\| \cdot \|_\Sigma$. For $u \in H^1 (\Omega)$ we denote by $u |_{\Sigma}$ the trace of $u$ on $\Sigma$, which extends the restriction map of smooth functions to $\Sigma$ as a bounded linear operator from $H^1 (\Omega)$ to $L^2 (\Sigma)$.

For our purposes it is convenient to deal with the Laplacian as well as the normal derivatives of appropriate Sobolev functions in the following weak sense; such definitions can be found, e.g., in the textbook~\cite{McL}.

\begin{dfn}\label{definition}
Let $\Omega \subset \dR^d$ be a Lipschitz domain.
\begin{itemize}
 \item [{\rm (i)}] Let $u \in H^1(\Omega)$. If there exists $f \in L^2 (\Omega)$ with
\[
(\nabla u, \nabla v)_{\Omega} = (f, v)_{\Omega} \quad \text{for all}~ v\in H^1_0(\Omega),
\]
we say $\Delta u \in L^2 (\Omega)$ and set $- \Delta u := f$.
\item [{\rm (ii)}]  Let $u \in H^1(\Omega)$ with $\Delta u \in L^2 (\Omega)$. If there exists $b \in L^2(\Sigma)$ with
\[
(\nabla u,\nabla v)_{\Omega} - (-\Delta u,v)_{\Omega} = (b, v|_{\Sigma})_{\Sigma} \quad \text{for all}~v \in H^1(\Omega),
\]
we say $\partial_{\nu} u|_{\Sigma} \in L^2(\Sigma)$ and set $\partial_{\nu} u|_{\Sigma} := b$.
\end{itemize}
\end{dfn}

We remark that $\partial_{\nu} u|_{\Sigma}$ is unique if it exists. For each sufficiently smooth $u \in L^2 (\Omega)$ the function $\partial_{\nu} u|_{\Sigma}$ on $\Sigma$ is the usual derivative in the direction of the outer unit normal, which follows immediately from the first Green identity.

We call $\Sigma \subset \dR^d$ a {\em Lipschitz hypersurface} if $\Sigma$ coincides with the boundary of a Lipschitz domain $\Omega_1 \subset \dR^d$. In this case also $\Omega_2 := \dR^d \setminus \ov{\Omega_1}$ is a Lipschitz domain with the same boundary $\Sigma$, and~$\Sigma$ separates~$\dR^d$ into~$\Omega_1$ and~$\Omega_2$. Note that we do not require $\Omega_1$, $\Omega_2$, or $\Sigma$ to be connected; see, e.g., Figure~\ref{hyphyp} in Example~\ref{example:unconnected} below. 

For a Lipschitz hypersurface $\Sigma$ and the corresponding Lipschitz domains $\Omega_1$ and $\Omega_2$ as above we occasionally write a function $u \in L^2 (\dR^d)$ as $u = u_1 \oplus u_2$, where $u_j = u |_{\Omega_j}$, $j = 1, 2$, referring to the orthogonal decomposition $L^2 (\dR^d) = L^2 (\Omega_1) \oplus L^2 (\Omega_2)$.  Moreover, we write $\partial_{\nu_j} u_j |_\Sigma$, $j = 1, 2$, for the normal derivative of~$u_j$ in Definition~\ref{definition}~(ii). 

For the following definition cf.~\cite[Section 2.3]{BEL13}.

\begin{dfn}\label{definition2}
Let $\Sigma$ be a Lipschitz hypersurface which separates $\dR^d$ into two Lipschitz domains $\Omega_1$ and $\Omega_2$. 
Let $u = u_1 \oplus u_2 \in H^1(\dR^d)$ with $\Delta u_j \in L^2(\Omega_j)$, $j = 1, 2$. If there exists $\wt b \in L^2(\Sigma)$ such that
\[
\big(\nabla u,\nabla v\big)_{\dR^d} - \big((-\Delta u_1) \oplus(-\Delta u_2), v\big)_{\dR^d} = (\wt b, v|_{\Sigma})_{\Sigma} \quad \text{for all}~v \in H^1(\dR^d),
\]
we say $[\partial_\nu  u]_\Sigma  \in L^2(\Sigma)$ and set $[\partial_\nu u]_\Sigma := \wt b$. 
\end{dfn}

Note that $[\partial_\nu u]_{\Sigma}$ is unique if it exists; cf.~\cite[Section~2.3]{BEL13}. The interpretation of $[\partial_\nu u]_{\Sigma}$ is provided in the following lemma.

\begin{lem}
\label{lem:weakjump}
Let $\Sigma$ be a Lipschitz hypersurface which separates $\dR^d$ into two Lipschitz domains $\Omega_1$ and $\Omega_2$. Let $u = u_1 \oplus u_2 \in H^1(\dR^d)$ with $\Delta u_j \in L^2(\Omega_j)$ and $\partial_{\nu_j} u_j |_{\Sigma} \in L^2(\Sigma)$, $j = 1, 2$. Then $[\partial_\nu u]_{\Sigma}\in L^2(\Sigma)$ and
\[
[\partial_\nu u]_{\Sigma} = \partial_{\nu_1} u_1 |_{\Sigma} +  \partial_{\nu_2} u_2 |_{\Sigma}.
\]
\end{lem}

\begin{proof}
Let us fix an arbitrary $v\in H^1(\dR^d)$. Clearly $v_j\in H^1(\Omega_j)$ holds for $j = 1, 2$. Thus employing Definition~\ref{definition}~(ii) we get
\[
\begin{split}
& \big(\nabla u,\nabla v\big)_{\dR^d} - \big((-\Delta u_1) \oplus(-\Delta u_2), v \big)_{\dR^d}\\ 
&\quad = \Big[\big(\nabla u_1,\nabla v_1\big)_{\Omega_1} - (-\Delta u_1, v_1\big)_{\Omega_1}\Big] + \Big[\big(\nabla u_2,\nabla v_2\big)_{\Omega_2} - \big(-\Delta u_2, v_2\big)_{\Omega_2}\Big]\\
&\quad = (\partial_{\nu_1} u_1|_{\Sigma} +  \partial_{\nu_2}u_2|_{\Sigma}, v|_{\Sigma})_{\Sigma}
\end{split}
\]
and the claim follows from Definition~\ref{definition2}.
\end{proof}

%-------------------------------------------------------------
\subsection{Schr\"odinger operators with {\boldmath $\delta$} and {\boldmath $\delta'$}-interactions}
%-------------------------------------------------------------
\label{sec:def}

In this paragraph we recall the mathematically rigorous definitions of the self-adjoint Schr\"odinger operators with $\delta$ and $\delta'$-interactions supported on a Lipschitz hypersurface~$\Sigma$. For the required material on semibounded, closed sesquilinear forms and corresponding self-adjoint operators we refer the reader to~\cite[Chapter~VI]{Kato}.

\begin{dfn}\label{def:operators}
Let $\Sigma$ be a Lipschitz hypersurface which separates $\dR^d$ into two Lipschitz domains $\Omega_1$ and $\Omega_2$.
\begin{enumerate}
 \item[(i)] The Schr\"odinger operator $-\Delta_{\delta,\alpha}$ in $L^2(\dR^d)$ with a $\delta$-interaction supported on $\Sigma$ of strength $\alpha : \Sigma \to \dR$ with $\alpha \in L^\infty(\Sigma)$ is the unique self-adjoint operator in $L^2 (\dR^d)$ which corresponds to the densely defined, symmetric, lower semibounded, closed sesquilinear form 
 \begin{equation}\label{formdelta}
  \fra_{\delta,\alpha}[u,v] = (\nabla u,\nabla v)_{\dR^d} - ( \alpha u|_{\Sigma}, v|_{\Sigma} )_\Sigma, \quad \dom \fra_{\delta,\alpha} = H^1(\dR^d),
 \end{equation}
 (cf.~\cite[Section 2]{BEKS94} for $C^1$-smooth $\Sigma$ and \cite[Proposition 3.1]{BEL13} for the Lipschitz case).
 \item[(ii)] The Schr\"odinger operator $-\Delta_{\delta',\beta}$ in $L^2(\dR^d)$ with a $\delta'$-interaction supported on $\Sigma$ of strength $\beta : \Sigma \to \dR$ with $\beta^{-1}\in L^\infty(\Sigma)$ is the self-adjoint operator in $L^2 (\dR^d)$ which corresponds to the densely defined, symmetric, lower semibounded and closed sesquilinear form 
 \begin{equation}\label{formdelta'}
 \begin{split}
 \fra_{\delta',\beta}[u,v] &\!=\! (\nabla u_1,\nabla v_1)_{\Omega_1}\!+\! (\nabla u_2,\nabla v_2)_{\Omega_2}\!-\!
 \Big( \frac{1}{\beta} (u_1|_{\Sigma}\!-\!u_2|_{\Sigma}), (v_1|_{\Sigma}\!-\! v_2|_{\Sigma}) \Big)_\Sigma,\\
 \dom \fra_{\delta',\beta} &= H^1(\dR^d\setminus\Sigma),
 \end{split}
 \end{equation}
 (cf.~\cite[Proposition 3.1]{BEL13}).
\end{enumerate}
\end{dfn}

The actions and domains of the operators~$-\Delta_{\delta,\alpha}$ and~$-\Delta_{\delta',\beta}$ can be characterized in the following way, using the weak Laplacians and normal derivatives from Definition~\ref{definition} and Definition~\ref{definition2}.

\begin{prop}\cite[Theorem 3.3]{BEL13}
\label{thm:dom}
Let $\Sigma$ be a Lipschitz hypersurface which separates $\dR^d$ into two Lipschitz domains $\Omega_1$ and $\Omega_2$. Moreover, let $\alpha, \beta : \Sigma \to \dR$ be functions such that $\alpha,\beta^{-1}\in L^\infty(\Sigma)$. Then the self-adjoint operators~$-\Delta_{\delta,\alpha}$ and~$-\Delta_{\delta',\beta}$ in Definition~\ref{def:operators} have the following representations.
\begin{itemize}
 \item [{\rm (i)}] 
 $-\Delta_{\delta,\alpha}u \!=\! (-\Delta u_1)\!\oplus\!(-\Delta u_2)$ 
 and $u\!=\! u_1\oplus u_2\in\dom(-\Delta_{\delta,\alpha})$ 
  \!\!\!  if and only if
\vskip 0.2cm
 \begin{itemize}\setlength{\itemsep}{0.2cm}
\item[\rm (a)] $u\in H^1(\dR^d)$,
\item[\rm (b)] $\Delta u_j  \in L^2(\Omega_j)$, $j = 1, 2$, and
\item[\rm (c)] $[\partial_\nu u]_{\Sigma}\in L^2(\Sigma)$
exists in the sense of Definition~\ref{definition2} and 
\[
[\partial_\nu u]_{\Sigma} = \alpha u|_{\Sigma}.
\]
\end{itemize}
 \item [{\rm (ii)}] 
 $-\Delta_{\delta^\prime,\beta}u\! =\! (-\Delta u_1)\!\oplus\! 
 (-\Delta u_2)$ and $u\!=\! u_1\oplus u_2\!\in\!\dom(\!-\Delta_{\delta^\prime,\beta})$ \!\!\!
  if and only if
\vskip 0.2cm
\begin{itemize}\setlength{\itemsep}{0.2cm}
\item[\rm (a$^\prime$)] $u_j \in H^1(\Omega_j)$, $j = 1, 2$,
\item[\rm (b$^\prime$)] $\Delta u_j  \in L^2(\Omega_j)$, $j = 1, 2$, and
\item[\rm (c$^\prime$)] 
$\partial_{\nu_j} u_j |_{\Sigma}\in L^2(\Sigma)$ exist
in the sense of Definition~\ref{definition}~(ii), $j = 1, 2$, and 
\begin{equation*}
  u_1|_{\Sigma} - u_2|_{\Sigma} = \beta \partial_{\nu_1} u_1|_{\Sigma} = - \beta \partial_{\nu_2} u_2 |_\Sigma.
\end{equation*}
\end{itemize}
\end{itemize}
\end{prop}

\section{Proof of Theorem~A}\label{sec:proof}

In this section we provide the proof of Theorem~A. As a first step we show the following proposition. In its formulation the unitary operator 
\begin{equation}
\label{U}
 U \colon L^2(\dR^d)\rightarrow L^2(\dR^d),\qquad U (u_1 \oplus u_2) := u_1 \oplus (-u_2),
\end{equation}
appears, which was already mentioned in the introduction.

\begin{prop}\label{prop:halbesTheorem}
Let $0 < \beta \leq 4 / \alpha$. If
\begin{equation}\label{Wmu}
 W_\mu :=  U \big(\dom(-\Delta_{\delta,\alpha})\big)\cap \ker(-\Delta_{\delta', \beta} -\mu) = \{0\}
\end{equation}
holds for each $\mu < \inf \sigma_{\rm ess} (- \Delta_{\delta', \beta})$ then
\[
\lambda_n(-\Delta_{\delta', \beta}) < \lambda_n(-\Delta_{\delta,\alpha})
\]
holds for all $n \in \dN$ such that $\lambda_n (- \Delta_{\delta, \alpha}) < \inf \sess (- \Delta_{\delta', \beta})$.
\end{prop}

\begin{proof}
Let $N_{\delta,\alpha}(\cdot)$ and $N_{\delta', \beta}(\cdot)$ be the counting functions for the eigenvalues below the bottom of the essential spectrum of the operators $-\Delta_{\delta,\alpha}$ and $-\Delta_{\delta', \beta}$, respectively, that is,
\begin{align*}
 N_{\delta,\alpha}(\mu) := \# \big\{ k \in \dN : \lambda_k (-\Delta_{\delta, \alpha}) \leq \mu \big\}, \quad \mu < \inf \sess (- \Delta_{\delta, \alpha}),
\end{align*}
and
\begin{align*}
 N_{\delta',\beta}(\mu) := \# \big\{ k \in \dN : \lambda_k (-\Delta_{\delta', \beta}) \leq \mu \big\}, \quad \mu < \inf \sess (- \Delta_{\delta', \beta}).
\end{align*}
It follows from the min-max principle, see~\cite[Chapter~10]{BS} or~\cite[Chapter~12]{S}, that these functions can be expressed as
\begin{align*}
 N_{\delta,\alpha}(\mu) = \max \left\{\dim L\colon L~\text{subspace of}~H^1(\dR^d),~\fra_{\delta,\alpha}[u] \le \mu\|u\|^2_{\dR^d},~u\in L \right\}
\end{align*}
and
\begin{align*}
 N_{\delta', \beta}(\mu) &= \max \left\{\dim L\colon\! L~\text{subspace of}~H^1(\dR^d\setminus\Sigma),~\fra_{\delta', \beta}[u] \le \mu\|u\|^2_{\dR^d},\!~u\in L \right\},
\end{align*}
where $\fra_{\delta,\alpha}$ and $\fra_{\delta', \beta}$ are the sesquilinear forms in~\eqref{formdelta} and~\eqref{formdelta'}, respectively. Let $\mu < \inf \sess (-\Delta_{\delta', \beta}) \leq \inf \sess (- \Delta_{\delta, \alpha})$ and define
\[
 F := U({\rm span}\{ \ker(-\Delta_{\delta,\alpha} - \lambda)\colon \lambda\le\mu\})
\]
with $U$ as in \eqref{U}. Then $\dim F = N_{\delta,\alpha}(\mu)$ and
\begin{equation}\label{I}
 \fra_{\delta,\alpha}[U^{-1} u] \le \mu\|u\|_{\dR^d}^2,\qquad u\in F,
\end{equation}
where we have used the abbreviation $\fra_{\delta, \alpha} [w] := \fra_{\delta, \alpha} [w, w]$ for $w \in \dom \fra_{\delta, \alpha}$. For $u\in F$ and $v\in \ker(-\Delta_{\delta', \beta}-\mu)$ we have $u_1 |_\Sigma = - u_2 |_\Sigma$ and it follows from~\eqref{formdelta'} that $u, v \in \dom \fra_{\delta', \beta}$ and
\begin{equation}
\label{frauv}
\begin{split}
\fra_{\delta', \beta}[u + v] &= \|\nabla (u_1 + v_1) \|^2_{\Omega_1} \!+\! \|\nabla (u_2 + v_2)\|^2_{\Omega_2} \\
 & \quad - \Big( \frac{1}{\beta} (2 u_1 |_{\Sigma} + (v_1 |_{\Sigma} -v_2 |_{\Sigma})), 2 u_1 |_{\Sigma} + (v_1 |_{\Sigma} -v_2 |_{\Sigma}) \Big)_\Sigma \\
 & = I+ J+ K,
\end{split}
\end{equation}
where
\[
\begin{split}
I &:= \|\nabla v_1 \|^2_{\Omega_1} +\|\nabla v_2\|^2_{\Omega_2} - \Big( \frac{1}{\beta} (v_1 |_{\Sigma} -v_2 |_{\Sigma}), v_1 |_{\Sigma} -v_2 |_{\Sigma} \Big)_\Sigma,\\
J  &:= \|\nabla u_1\|^2_{\Omega_1} +\|\nabla u_2\|^2_{\Omega_2} - \Big( \frac{4}{\beta} u_1|_{\Sigma}, u_1|_{\Sigma} \Big)_\Sigma,
\end{split}
\]
and
\[
 K := 2\Re\Big[ \big( \nabla u_1, \nabla v_1 \big)_{\Omega_1} + \big( \nabla u_2, \nabla v_2 \big)_{\Omega_2}
 - \Big( \frac{2}{\beta} u_1 |_{\Sigma}, (v_1 |_{\Sigma} - v_2|_{\Sigma}) \Big)_\Sigma \Big].
\]
According to the choices of $u$ an $v$ and due to \eqref{I} we get
\begin{equation}\label{IJ}
 I \!= \fra_{\delta', \beta} [v] = \mu\|v\|_{\dR^d}^2\qquad\text{and}\qquad J\! = \fra_{\delta, 4/\beta} [U^{-1} u] \le \fra_{\delta,\alpha}[U^{-1} u] \le \mu\|u\|^2_{\dR^d},
\end{equation}
since $\alpha \leq 4/\beta$. Moreover, Definition~\ref{definition}~(ii) and Proposition~\ref{thm:dom}~(ii) give us
\[
 K  = 2 \Re \Big[ \mu (u, v)_{\dR^d} + \big( u_1 |_{\Sigma}, \partial_{\nu_1} v_1|_{\Sigma} \big)_\Sigma + 
\big( u_2 |_{\Sigma}, \partial_{\nu_2} v_2 |_{\Sigma} \big)_\Sigma - \Big(\frac{2}{\beta} u_1 |_{\Sigma}, (v_1 |_{\Sigma} -v_2 |_{\Sigma}) \Big)_\Sigma\!\Big],
\]
where we have used that $-\Delta v_j = \mu v_j$, $j = 1, 2$. By Proposition~\ref{thm:dom} we have
\[
u_1 |_{\Sigma} = -u_2 |_{\Sigma}\quad\text{and}\quad \partial_{\nu_1} v_1 |_{\Sigma} =-\partial_{\nu_2} v_2 |_{\Sigma} = \frac{1}{\beta}(v_1|_{\Sigma} -v_2|_{\Sigma}),
\]
and hence we  obtain
\[
\begin{split}
K & = 2\Re\Big[\mu (u, v)_{\dR^d} + \Big( \frac{2}{\beta} u_1 |_{\Sigma}, (v_1 |_{\Sigma} -v_2 |_{\Sigma}) \Big)_\Sigma - \Big( \frac{2}{\beta} u_1 |_{\Sigma}, (v_1 |_{\Sigma} -v_2 |_{\Sigma}) \Big)_\Sigma \Big]\\
&  =2\mu \Re (u, v)_{\dR^d}.
\end{split}
\]
Combining the above expression for $K$ with~\eqref{frauv} and~\eqref{IJ} we arrive at
\begin{align}\label{eq:est}
\fra_{\delta', \beta}[u+v] \le 
\mu\|u \|^2_{\dR^d}
+ 2\mu \Re (u, v)_{\dR^d} + \mu\|v \|^2_{\dR^d}= \mu\|u + v\|^2_{\dR^d}
\end{align}
for all $u \in F$ and all $v \in \ker (- \Delta_{\delta', \beta} - \mu)$. From the assumption~\eqref{Wmu} we conclude 
\begin{align*}
 \dim \big( F + \ker (- \Delta_{\delta', \beta} - \mu) \big) = N_{\delta,\alpha}(\mu) + \dim \ker (- \Delta_{\delta', \beta} - \mu)
\end{align*}
and thus~\eqref{eq:est} implies
\[
N_{\delta', \beta}(\mu) \ge 
N_{\delta,\alpha}(\mu) + 
\dim\ker(-\Delta_{\delta', \beta} - \mu).
\]
Hence,
\begin{align*}
\# \big\{ k\in\dN\colon  \lambda_k(-\Delta_{\delta', \beta}) < \mu\big\}
= N_{\delta', \beta}(\mu)
-\dim\ker(-\Delta_{\delta', \beta} - \mu) 
\ge N_{\delta,\alpha}(\mu).
\end{align*}
Choosing $\mu = \lambda_n (-\Delta_{\delta,\alpha})$ for an arbitrary $n \in\dN$ such that $\mu < \inf \sigma_{\rm ess} (- \Delta_{\delta', \beta})$, it follows
\[
 \# \big\{k\in\dN\colon  \lambda_k(-\Delta_{\delta', \beta}) < \lambda_n (-\Delta_{\delta,\alpha}) \big\}
\ge n.
\]
Thus 
$\lambda_n(-\Delta_{\delta', \beta}) < \lambda_n(-\Delta_{\delta,\alpha})$ for all $n$ with $\lambda_n (- \Delta_{\delta, \alpha}) < \inf \sess (- \Delta_{\delta', \beta})$. This completes the proof of the proposition.
\end{proof}

We will now apply Proposition~\ref{prop:halbesTheorem} in order to prove Theorem~A.

\begin{proof}[{\bf Proof of Theorem~A}]
Let $\sigma \subset \Sigma$ be a nonempty open set such that 
\begin{align}\label{eq:assumpt}
 \beta |_\sigma < (4 / \alpha) |_\sigma.
\end{align}
By Proposition~\ref{prop:halbesTheorem}, in order to prove Theorem~A it suffices to verify~\eqref{Wmu} for each $\mu  < \inf \sess (- \Delta_{\delta', \beta})$. Let us fix such a~$\mu$ and let~$u \in W_\mu$. Proposition~\ref{thm:dom}~(ii) yields 
\begin{align}\label{eq:solutions}
 - \Delta u_j = \mu u_j, \quad j = 1, 2,
\end{align}
and
\begin{equation}\label{eq:bc1}
 u_1|_{\Sigma} - u_2|_{\Sigma} = \beta \partial_{\nu_1}u_1|_{\Sigma} = - \beta \partial_{\nu_2} u_2|_{\Sigma}.
\end{equation}
On the other hand, from Proposition~\ref{thm:dom}~(i) and Lemma~\ref{lem:weakjump} we obtain
\begin{equation}\label{eq:bc2}
 u_1 |_{\Sigma} + u_2 |_{\Sigma} = 0,\quad \text{and} \quad \partial_{\nu_1} u_1 |_{\Sigma} - \partial_{\nu_2} u_2 |_{\Sigma} = \alpha u_1 |_{\Sigma}.
\end{equation}
The conditions \eqref{eq:bc1} and \eqref{eq:bc2} yield
\begin{align}\label{eq:contradiction}
 \partial_{\nu_1} u_1 |_\Sigma = \alpha u_1 |_\Sigma + \partial_{\nu_2} u_2 |_\Sigma = \frac{\alpha \beta}{2} \partial_{\nu_1} u_1 |_\Sigma + \partial_{\nu_2} u_2 |_\Sigma = \bigg( \frac{\alpha \beta}{2} - 1 \bigg) \partial_{\nu_1} u_1 |_\Sigma. 
\end{align}
By~\eqref{eq:assumpt} we have $\frac{\alpha \beta}{2} |_\sigma - 1 < 1$ on~$\sigma$, hence~\eqref{eq:contradiction} implies $\partial_{\nu_1} u_1 |_\sigma = 0$. With the help of~\eqref{eq:bc1} and~\eqref{eq:bc2} it follows $u_1 |_\sigma = \frac{\beta}{2} \partial_{\nu_1} u_1 |_\sigma = 0$. Let now $\Omega$ be a connected component of $\Omega_1$ such that $\partial \Omega \cap \sigma \neq \varnothing$. As in the proof of~\cite[Proposition~2.5]{BR12} let us choose a connected Lipschitz domain $\widetilde \Omega$ such that $\Omega \subset \widetilde \Omega$, $\partial \Omega \setminus \sigma \subset \partial\widetilde \Omega$, and $\widetilde \Omega \setminus \Omega$ has a nonempty interior. Then the function $\widetilde u$ with $\widetilde u = u_1$ on $\Omega$ and $\widetilde u = 0$ on $\widetilde \Omega \setminus \Omega$ belongs to $L^2 (\widetilde \Omega)$ and satisfies $- \Delta \widetilde u = \mu \widetilde u$ on $\widetilde \Omega$. Indeed, $\widetilde u \in H^1 (\widetilde \Omega)$ 
since $u_1 |_\sigma = 0$. Moreover, for each $\widetilde v \in H_0^1 (\wt\Omega)$ we have
\begin{align*}
 (\nabla \widetilde u, \nabla \widetilde v)_{\widetilde \Omega} = (\nabla u_1, \nabla v)_{\Omega} = (- \Delta u_1, v)_{\Omega} + (\partial_{\nu_1} u_1 |_{\partial \Omega}, v |_{\partial \Omega} )_{\partial \Omega},
\end{align*}
where $v$ denotes the restriction of $\wt v$ to $\Omega$. Since $v |_{\partial \Omega \setminus \sigma} = 0$  and $\partial_{\nu_1} u_1 |_\sigma = 0$ it follows with the help of~\eqref{eq:solutions}
\begin{align*}
 (\nabla \widetilde u, \nabla \widetilde v)_{\widetilde \Omega} = (\mu u_1, v)_{\Omega} = (\mu \widetilde u, \widetilde v)_{\widetilde \Omega},
\end{align*}
thus $- \Delta \widetilde u = \mu \widetilde u$ by Definition~\ref{definition}~(i). As $\widetilde u$ vanishes on the nonempty interior of $\widetilde \Omega \setminus \Omega$, a unique continuation argument implies $\widetilde u = 0$, see, e.g.,~\cite[Theorem~XIII.63]{RS}. Hence $u_1$ is identically equal to zero on the connected component $\Omega$ of $\Omega_1$. 

It remains to conclude from this that $u = 0$ identically on $\dR^d$. Indeed, since $\Sigma$ separates $\dR^d$ into the Lipschitz domains $\Omega_1$ and $\Omega_2$, there exists a connected component $\Lambda$ of $\Omega_2$ such that $\tau := \partial \Omega \cap \partial \Lambda \neq \varnothing$. Since $u_1 |_{\tau} = \partial_{\nu_1} u_1 |_{\tau} = 0$ it follows with the help of~\eqref{eq:bc1} that $u_2 |_{\tau} = \partial_{\nu_2} u_2 |_{\tau} = 0$; another application of unique continuation implies $u_2 |_\Lambda = 0$. Repeating the same argument successively for the respective neighboring connected components finally it follows $u = 0$ on all of $\dR^d$, which completes the proof of the theorem.
\end{proof}

\section{The borderline case $\beta = 4/\alpha$}\label{sec:examples}

In this section we present various examples with explicit geometries of the interaction support~$\Sigma$, where $\beta = 4 / \alpha$ and the strict eigenvalue inequality in Theorem~A remains valid. In all the following examples the strengths of interactions $\alpha$ and $\beta$ are constants.

\begin{example}\label{example:broken}
In this example we consider the broken line
\begin{equation*}
% \label{Sigma}
\Sigma := \big\{(x, \cot(\theta)|x|)\in\dR^2 \colon x\in\dR\big\},
\qquad \theta\in (0,\pi/2),
\end{equation*}
which splits $\dR^2$ into the two domains
\begin{align*}
 \Omega_1 = \big\{(x, y)\in\dR^2 \colon x\in\dR, y > \cot(\theta)|x| \big\}
\end{align*}
and
\begin{align*}
 \Omega_2 = \big\{(x, y)\in\dR^2 \colon x\in\dR, y < \cot(\theta)|x|\big\};
\end{align*}
cf. Figure~\ref{fig:broken}.
\begin{figure}[h]
\vspace{3ex}
\begin{center}
\begin{picture}(120,100)
\put(80,67){$\Sigma$}
\put(51,67){$\Omega_1$}
\put(15,25){$\Omega_2$}
\put(50,20){\line(1,2){35}}
\put(50,20){\line(-1,2){35}}
\multiput(50,20)(0,7){12}{\line(0,1){5}}
\qbezier(50,30)(52,33)(54,28)
\put(52,35){$\theta$}
\end{picture}
\end{center}
\caption{A broken line $\Sigma$ with angle $\theta\in(0,\pi/2)$,
which splits $\dR^2$ into two wedge-type domains $\Omega_1$ and $\Omega_2$.}
\label{fig:broken}
\end{figure}
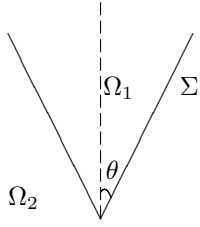
Moreover, we assume that $\beta = 4 / \alpha > 0$ is constant. Then
\begin{align*}
 \sess(-\Delta_{\delta,\alpha}) = \big[-\alpha^2/4,+\infty\big) = \big[-4/\beta^2,+\infty) = \sess(-\Delta_{\delta',\beta}),
\end{align*}
see~\cite[Proposition 5.4]{EN03} and~\cite[Corollary~4.11]{BEL13}, and the discrete spectra of both operators are nonempty, see~\cite[Theorem 5.2]{EI01} and~\cite[Corollary~4.12]{BEL13}.
 
We are going to apply~Proposition~\ref{prop:halbesTheorem}. Let $\mu < -\alpha^2/4$ and~$u \in W_\mu$, see~\eqref{Wmu}. By Proposition~\ref{thm:dom}\,(ii) we have
\begin{equation}
\label{bc1}
 u_1|_{\Sigma} - u_2|_{\Sigma} = 
(4/\alpha) \partial_{\nu_1}u_1|_{\Sigma} = -(4/\alpha) \partial_{\nu_2}u_2|_{\Sigma},
\end{equation}
and from Proposition~\ref{thm:dom}\,(i) and Lemma~\ref{lem:weakjump} we obtain
\begin{equation}\label{bc2}
 u_1|_{\Sigma} + u_2|_{\Sigma} = 0 \quad \text{and} \quad \partial_{\nu_1}u_1|_{\Sigma} - \partial_{\nu_2}u_2|_{\Sigma} = \alpha u_1|_{\Sigma}.
\end{equation}
Combining~\eqref{bc1} and~\eqref{bc2} yields
\begin{equation}\label{bc3}
 \partial_{\nu_j}u_j|_{\Sigma} = (\alpha/2)u_j|_{\Sigma},\qquad j = 1,2.
\end{equation}
It was shown in~\cite[Lemma~2.8]{LP08} that the bottom of the spectrum of the self-adjoint Laplacian in~$L^2 (\Omega_2)$ subject to the Robin boundary condition~\eqref{bc3} equals~$- \alpha^2 / 4$. Since $- \Delta u_2 = \mu u_2$ on~$\Omega_2$ and $\mu < - \alpha^2/4$, it follows $u_2 = 0$ identically. Plugging this into~\eqref{bc2} implies $u_1|_{\Sigma} = 0$. Recall that $\mu < -\alpha^2/4$ and that the function $u_1$ satisfies $-\Delta u_1 = \mu u_1$ in $\Omega_1$. Since the self-adjoint Dirichlet Laplacian on $\Omega_1$ is non-negative, we get $u_1 = 0$ identically as well, hence $u = 0$. Thus it follows from Proposition~\ref{prop:halbesTheorem} that
\begin{align*}
 \lambda_n (- \Delta_{\delta', \beta}) < \lambda_n (- \Delta_{\delta, \alpha})
\end{align*}
holds for all $n \in \dN$ such that $\lambda_n (- \Delta_{\delta, \alpha}) < - \alpha^2/4$.
\end{example}

\begin{example}
Another example of a similar flavour is given by the cone
\begin{equation}
\label{Sigma2}
 \Sigma := \big\{(x,y, \cot(\theta)\sqrt{x^2+y^2})\in\dR^3 \colon (x,y)\in\dR^2\big\},
\qquad \theta\in (0,\pi/2);
\end{equation}
cf.~Figure~\ref{fig:cone}.
\begin{figure}[h]
\vspace{2ex}
\begin{center}
\begin{picture}(120,100)
\includegraphics[scale = 0.4]{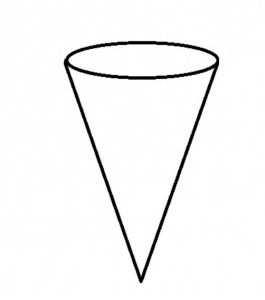}
\put(-12,65){$\Sigma$}
\multiput(-37,10)(0,7){14}{\line(0,1){5}}
\qbezier(-36.5,30)(-32,33)(-29,28)
\put(-33.5,37){$\theta$}
\end{picture}
\end{center}
\caption{An infinite cone $\Sigma$ with angle $\theta\in(0,\pi/2)$.}
\end{figure}
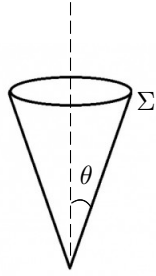
For constant $\alpha > 0$ it was shown in~\cite[Theorem~2.1]{BEL14} that
$\sess(-\Delta_{\delta,\alpha}) = [-\alpha^2/4,+\infty)$,
and the discrete spectrum of $-\Delta_{\delta,\alpha}$ was proved in~\cite[Theorem~3.2]{BEL14}
to be nonempty and even infinite. Following the lines of Example~\ref{example:broken} and referring to~\cite[Example~2.9]{LP08} instead of~\cite[Lemma~2.8]{LP08} it follows for constant $\beta = 4 / \alpha > 0$
\begin{align}\label{eq:IneqCone}
 \lambda_n(-\Delta_{\delta', \beta}) < \lambda_n(-\Delta_{\delta,\alpha})
\end{align}
for all $n \in \dN$ such that $\lambda_n (- \Delta_{\delta, \alpha}) < \inf \sess (- \Delta_{\delta', \beta})$.\footnote{In fact we expect that one can prove $\inf\sess(-\Delta_{\delta',\beta}) = -4/\beta^2$ using the arguments in the proof of~\cite[Theorem 2.1]{BEL14}. This would imply that~\eqref{eq:IneqCone} holds for all $n \in \dN$.}
\label{fig:cone}
\end{example}

\begin{example}\label{example:unconnected}
In this example we consider an unconnected hypersurface~$\Sigma$. Let $\dR^d_\pm := \{(x',x_d)\colon x'\in\dR^{d-1}, x_d \in\dR_\pm\}$, let $\Omega'\subset\dR^d_+$ be a bounded Lipschitz domain with positive distance to $\dR^d_-$ and let
\begin{equation}\label{Sigma3}
 \Sigma := \{(x',0)\colon x'\in\dR^{d-1}\}\cup\partial\Omega'.
\end{equation}
The surface~$\Sigma$ splits $\dR^d$ into the two Lipschitz domains
\begin{align*}
 \Omega_1 = \Omega' \cup \dR^d_- \quad \text{and} \quad \Omega_2 = \dR^d_+ \setminus \overline{\Omega'};
\end{align*}
cf. Figure~\ref{hyphyp}.
\begin{figure}[h]
\begin{center}
\begin{picture}(200,120)
\qbezier(60,85)(85,115)(110,90)
\qbezier(110,90)(125,70)(100,60)
\qbezier(100,60)(85,50)(75,70)
\qbezier(75,70)(70,75)(65,75)
\qbezier(60,85)(55,77.5)(65,75)
\put(85,78){$\Omega_1$}
\put(150,70){$\Omega_2$}
%\put(50,95){$\partial\Omega_{1}$}
%\put(170,50){$\Sigma_{23}$}
\qbezier(0,45)(100,45)(200,45)
\put(100,18){$\Omega_1$}
\put(30,35){$\Sigma$}
\put(49,75){$\Sigma$}
%--
\qbezier(0,0)(0,60)(0,120)
\qbezier(200,0)(200,60)(200,120)
\qbezier(0,0)(100,00)(200,0)
\qbezier(0,120)(100,120)(200,120)
\put(2,110){$\dR^d$}
\end{picture}
\end{center}
\caption{The unconnected hypersurface~$\Sigma$ splits $\dR^d$ into two domains $\Omega_1$ and~$\Omega_2$, and $\Omega_1$ consists of two connected components.}
\label{hyphyp}
\end{figure}
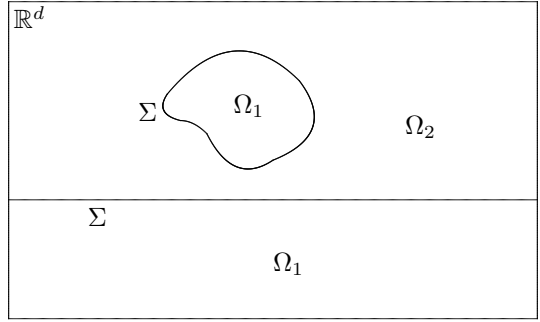
As in the previous examples we consider constant interaction strengths $\alpha, \beta$ with $\beta = 4 / \alpha > 0$. According to~\cite[Corollary~4.9]{BEL13} for constants $\alpha,\beta >0$ we have
\begin{align*}
 \sess(-\Delta_{\delta,\alpha}) = \big[-\alpha^2/4,+\infty \big) = \big[ -4/\beta^2,+\infty \big) = \sess(-\Delta_{\delta',\beta}).
\end{align*}
We are going to conclude from Proposition~\ref{prop:halbesTheorem} that
\begin{align}\label{eq:IneqDiscon}
 \lambda_n(-\Delta_{\delta', \beta}) < \lambda_n(-\Delta_{\delta,\alpha})
\end{align}
holds for all $n \in \dN$ such that $\lambda_n (- \Delta_{\delta, \alpha}) < - \alpha^2 / 4$. In order to do so, let $\mu < -\alpha^2/4$ and~$u \in W_\mu$ with $W_\mu$ as in \eqref{Wmu}. As in Example~\ref{example:broken} we find that $u$ satisfies the conditions~\eqref{bc1},~\eqref{bc2}, and~\eqref{bc3}. Since the spectrum of the self-adjoint Laplacian on $\dR^d_-$ satisfying the Robin boundary condition~\eqref{bc3} equals~$[-\alpha^2/4,+\infty)$, we conclude from~\eqref{bc3} and $- \Delta u_1 = \mu u_1$ that $u_1 |_{\dR^d_-}= 0$ identically. Together with~\eqref{bc2} and a unique continuation argument it follows as in the proof of Theorem~A that $u_2 = 0$ identically on $\Omega_2$. Finally, after another application of~\eqref{bc3} and of the unique continuation principle we arrive at $u_1 |_{\Omega'} = 0$, hence $u = 0$. Therefore Proposition~\ref{prop:halbesTheorem} yields the eigenvalue inequality~\eqref{eq:IneqDiscon}.
\end{example}
%
%\begin{remark}
%Although we were able to prove strict eigenvalue inequalities in the borderline case $0 < \beta = 4/\alpha$ only for
%special geometries, we are unaware of any example  in which for the borderline case strict eigenvalue inequality fails to hold.\marginpar{Koennte fuer den Gutachter ein Kritikpunkt sein.}
%\end{remark}

\end{document}